\documentclass[12pt,twoside]{article}

\usepackage{a4,enumerate}
\usepackage[noadjust]{cite}
\usepackage{amssymb,amsmath,amsthm}
\usepackage{comment}

\newcommand\mytitle{Dirichlet forms for singular diffusion
on graphs} 
\newcommand\lhead{C. Seifert, J. Voigt}
\newcommand\rhead{Dirichlet forms on graphs}
\swapnumbers
\numberwithin{equation}{section}

\newtheorem{theorem}{Theorem}[section]
\newtheorem{corollary}[theorem]{Corollary}

\newtheorem{lemma}[theorem]{Lemma}
\theoremstyle{definition}

\newtheorem{remark}[theorem]{Remark}
\newtheorem{remarks}[theorem]{Remarks}
\newtheorem{example}[theorem]{Example}

 \mathchardef\ordinarycolon\mathcode`\:
  \mathcode`\:=\string"8000
  \begingroup \catcode`\:=\active
    \gdef:{\mathrel{\mathop\ordinarycolon}}
  \endgroup

\newcommand\llim{
\mathchoice{\vcenter{\hbox{${\scriptstyle{-}}$}}}
{\vcenter{\hbox{$\scriptstyle{-}$}}}
{\vcenter{\hbox{$\scriptscriptstyle{-}$}}}
{\vcenter{\hbox{$\scriptscriptstyle{-}$}}}}
\newcommand\rlim{
\mathchoice{\vcenter{\hbox{${\scriptstyle{+}}$}}}
{\vcenter{\hbox{$\scriptstyle{+}$}}}
{\vcenter{\hbox{$\scriptscriptstyle{+}$}}}
{\vcenter{\hbox{$\scriptscriptstyle{+}$}}}}

\newcommand\smid{\nonscript \mskip2mu plus2mu {\mid}%
\nonscript \mskip2mu plus2mu}     
\def\scpr(#1,#2){{(#1\smid#2)}}
\def\bigscpr(#1,#2){{\bigl(#1\nonscript \mskip2mu plus2mu \big|\nonscript \mskip2mu
plus2mu#2\bigr)}}

\newcommand\vtx{\gamma}
\newcommand\KEV{\K^{E_1'\cup V_1}}
\newcommand\rmc{{\rm c}}

\newcommand\trace{\operatorname{tr}}
\newcommand\strace{\operatorname{str}}
\newcommand\bv{\mathop{{BV}}\nolimits}
\newcommand\ran{R}

\newcommand\indic{\mathop{\bf 1}\nolimits}

\newcommand{\spt}{\operatorname{spt}}

\newcommand{\lin}{\operatorname{lin}}
\newcommand{\proj}{\operatorname{pr}}

\newcommand\imu{{\rm i}}
\renewcommand\phi{\varphi}

\renewcommand\epsilon{\varepsilon}
\newcommand\1{{\bf 1}}

\newcommand{\R}{\mathbb{R}\nonscript\hskip.03em}

\newcommand{\C}{\mathbb{C}\nonscript\hskip.03em}
\newcommand{\K}{\mathbb{K}\nonscript\hskip.03em}

\def\formE(#1,#2){\sum_{e\in E}\int_{a_e}^{b_e} #1_e'(x)\ol{#2_e'(x)}\,dx}

\makeatletter
\let\qedhere@ams\qedhere
\def\qedhere{\@ifnextchar[{\@qedhere}{\qedhere@ams}}
\def\@qedhere[#1]{\tag*{\raisebox{-#1ex}{\qedhere@ams}}}
\makeatletter
\def\env@cases{%
  \let\@ifnextchar\new@ifnextchar
  \left\lbrace
  \def\arraystretch{1.1}%
  \array{@{\,}l@{\quad}l@{}}%
}
\renewcommand\section{\@startsection {section}{1}{\z@}%
                                     {-3.25ex \@plus -1ex \@minus -.2ex}%
                                     {1.5ex \@plus.2ex}%
                                     {\normalfont\large\bfseries}}
\makeatother

\newcommand\restrict[1]{{\left.\vphantom{f}\vrule\right._{#1}}}
\newcommand\set[2]{\bigl\{#1{;}\;#2\bigr\}}
\newcommand\sset[2]{\{#1{;}\;#2\}}
\newcommand\ol{\overline}

\renewcommand\le{\leqslant}

\renewcommand\ge{\geqslant}
\renewcommand\geq{\geqslant}

\newcommand{\from}{{:}\;}
\renewcommand{\from}{\colon}

\newcommand\sse{\subseteq}
\newcommand\sms{\setminus}

\newcommand\cH{\mathcal{H}}

\newcommand\cHg{{\mathcal H_\Gamma}}
\newcommand\cHe{{\mathcal H_{E}}}

\newcommand{\abstracttext}{\noindent
We describe operators driving the time evolution of singular diffusion on
finite graphs whose vertices are allowed to carry masses. The operators are
defined by the method of quadratic forms on suitable Hilbert spaces. The model
also covers quantum graphs and discrete Laplace operators.

\vspace{8pt}

\noindent
MSC 2010: 47D06, 60J60, 47E05, 35Q99, 05C99
\vspace{2pt}

\noindent
Keywords: gap diffusion, quantum graph, Dirichlet form, $C_0$-semigroup,
positive, submarkovian
}
\begin{document}
\title{\mytitle}

\author{Christian Seifert and J\"urgen Voigt}

\date{}

\maketitle

\begin{abstract}
\abstracttext
\end{abstract}

\section*{Introduction}

The present paper is a continuation and extension of \cite{kkvw09}. We present
suitable boundary or glueing conditions on graphs (quantum graphs) with
singular second order differential operators on the edges. In particular, we
describe those
boundary
conditions leading to positive and submarkovian
$C_0$-semigroups.

The graph consists of finitely many bounded intervals, the edges, whose end
points are connected with the vertices of the graph. On each of the edges $e$ a
finite Borel measure $\mu_e$ is given, determining where particles may be
located. The particles  move according to ``Brownian motion'' but are slowed
down or accelerated by the ``speed measure'' $\mu_e$. Further, each of the
vertices $v$ is provided with a weight $\mu_v\ge0$, and particles may also be
located at those vertices $v$ with~$\mu_v>0$.

The motivations for the treatment in \cite{kkvw09} were twofold. 
The first issue
was to treat singular diffusion, including gap diffusion, on the
edges of the graph, in the framework of Dirichlet forms. 
The second
aim was to describe glueing conditions on the vertices, in the spirit of
\cite{kuc04}, and investigate conditions under which the associated 
self-adjoint operator gives rise to a positive or submarkovian $C_0$-semigroup.

In the present paper, the
extension with respect to \cite{kkvw09} consists in two issues. On the one
hand, the boundary conditions we describe are more general than 
glueing conditions. By glueing conditions or ``local boundary
conditions'', we understand conditions where, for a given vertex, only the
values of a function on adjacent edges and on the vertex itself can interact. In
our treatment in Sections~\ref{singDir} and \ref{operator}, however, the graph
structure does not
intervene at all, and we only
specify later the case of local boundary conditions, in Section \ref{lbc}. On
the other hand, we include the general case of vertices with masses, whereas in
\cite[Section 4]{kkvw09} only a special case was treated.
These results have been obtained in \cite{sei09}.
\medskip

The ultimate objective of the treatment is to obtain a semi-bounded (below)
self-adjoint operator $H$ on a Hilbert space $\cH_\Gamma$ over the graph 
$\Gamma$ which
can then be used in the initial value problem for
the diffusion equation or heat equation
\begin{align}\label{diffequ}
u'=-Hu,
\end{align}
thus governing the time evolution of a process, i.e.,
giving rise to a $C_0$-semigroup on $\cH_\Gamma$. For this equation 
it is of interest to obtain $H$ in such a way that the associated
$C_0$-semigroup is positive or submarkovian. The self-adjointness of $H$ is also
of interest for the initial value problem for the Schr\"odinger equation
\[
u'=-\imu Hu.
\]

The part of the operator $H$ acting on an edge $e$ is of the form
$(Hf)_e=-\partial_{\mu_e}\partial f_e$, where $\partial_{\mu_e}$ is the
derivative with respect to $\mu_e$; cf.\ Section~\ref{odp}. The domain of $H$ is
restricted by conditions
on the 
boundary values of the functions on the edges
and the values
at
the vertices.

The Hilbert space $\cH_\Gamma$ is given by
\[
\cH_\Gamma=\bigoplus_{e\in E} L_2([a_e,b_e],\mu_e)\oplus \K^V, 
\]
where $E$ is the set of edges, the interval $[a_e,b_e]$ corresponds to 
the edge $e$,
and $V$ is the set of vertices; cf.\ Section~\ref{singDir} for more details. The
operator $H$ is obtained by the method of forms. Avoiding all technicalities
(which will be given in Section~\ref{singDir}), the form $\tau$ giving rise to
$H$ is of the form
\[
\tau(f,g)=\sum_{e\in E} \int_{a_e}^{b_e}f'_e(x)\ol{g'_e(x)}\,dx
                               +\scpr(L\trace f,\trace g),
\]
with domain
\[
D(\tau)=\set{f\in\ldots}{\trace f\in X}.
\]		       
Here, $\trace f$ denotes the boundary values of $f$ 
on the edges and the values of $f$ on the vertices, $X$ is a subspace of the set
of possible boundary values and values on the vertices, and $L$ is a
self-adjoint operator (matrix) on $X$. The boundary conditions for functions in 
the domain of $H$ are encoded in the space $X$ as well as in the operator $L$;
cf. Theorem~\ref{thm-operator}. Our treatment includes the case that some of the
edges or vertices may have weight zero.

For the discussion of positivity and the submarkovian property in connection
with equation \eqref{diffequ} we use the Beurling-Deny criteria for $\tau$.
These yield the result that the subspace $X$ should satisfy lattice properties
and $L$ should satisfy positivity properties; 
cf.~Theorem~\ref{pos-subm}.

The investigations mentioned so far did not take into account the graph
structure of $\Gamma$. In the description of glueing conditions, allowing only
interactions between vertices and adjacent edges, the space $X$ and the
operator $L$ decompose into
parts corresponding to single vertices; cf.
Corollaries~\ref{cor-local1} and \ref{cor-local2}.

\medskip
In Section~\ref{odp} we recall some notation and facts from the one-dimensional
case on an interval. In Section~\ref{singDir} we define the form in the Hilbert
space $\cHg$ on the graph which then defines the operator driving the evolution.
We show that the defined form $\tau$ constitutes a form that is bounded
below and closed. Let us point out that our definition of the form looks
somewhat different from the one given in \cite[Section~3]{kkvw09}. In fact,
looking at the definition of $\tau$ in \cite[Section 3]{kkvw09}, one realises
that there is some interpretation needed in order to understand $D(\tau)$ as a
subset of the Hilbert space $\cHg$. This interpreation is made explicit in the
present paper by the use of the mapping $\iota$ introduced in
Sections~\ref{odp} and~\ref{singDir}. In Section~\ref{operator} we describe the
operator $H$ associated with the form $\tau$ (Theorem~\ref{thm-operator}). In
Section~\ref{pos-con} we indicate conditions for the $C_0$-semigroup
$(e^{-tH})_{t\ge0}$ to be positive and submarkovian. In Section~\ref{lbc} we
describe
the case of local boundary conditions.

\section{One-dimensional prerequisites}
\label{odp}

In order to define the classical Dirichlet form we have to recall some notation
and facts for a single interval $[a,b]\sse\R$, where $a,b\in\R$, $a<b$. Let
$\mu$ be a finite Borel measure on $[a,b]$, $a,b\in\spt\mu$, $\mu(\{a,b\})=0$. 
Our function spaces will consist of $\K$-valued functions, where 
$\K\in\{\R,\C\}$.
We define
\begin{align*}
 &C_\mu[a,b]:=\set{f\in C([a,b]}{f\text{ affine linear on the components of
}[a,b]\sms\spt\mu},\\
&W^1_{2,\mu}(a,b):= W^1_2(a,b)\cap C_\mu[a,b].
\end{align*}
For later use we recall the following inequalities. There exists a
constant $C>0$ such that
\begin{align}\label{sob-ine1}
\|f\|_\infty\le C\big(\|f'\|^2_{L_2(a,b)} +
\|f\|^2_{L_2([a,b],\mu)}\big)^{1/2}
\end{align}
for all $f\in W^1_2(a,b)\cap C[a,b]$, and for all $r\in(0,b-a]$ one has
\begin{align}\label{sob-ine2}
 |f(a)| 
\le r^{1/2}\|f'\|_{L_2(a,a+r)}+
			    \|f\|_{L_2([a,a+r],\mu)}\,\mu([a,a+r])^{-1/2},
\end{align}
and correspondingly for $b$; cf. \cite[Lemma 1.4 and Remark 3.2(b)]{kkvw09}.

Let $\kappa\colon W^1_2(a,b)\cap C[a,b]\to L_2([a,b],\mu)$ be defined by
$\kappa f:=f$. Then it can be shown that $\ran(\kappa)=\ran(\kappa\restrict
{W^1_{2,\mu}(a,b)})$ (cf. \cite[Lemma 1.2]{kkvw09}), and that $\kappa\restrict
{W^1_{2,\mu}(a,b)}$ is injective (cf. \cite[lower part of p.\,639]{kkvw09}). We
define $\iota:=\big(\kappa\restrict{W^1_{2,\mu}(a,b)}\big)^{-1}$. Thus, $\iota$
is an operator from $L_2([a,b],\mu)$ to $W^1_{2,\mu}(a,b)$,
\begin{align*}
 D(\iota)=\set{f\in L_2([a,b],\mu)}{\text{there exists }&g\in W^1_{2}(a,b)\cap
C[a,b] \\&\text{ such that } g=f\ \mu\text{-a.e.}},
\end{align*}
and $\iota f$ is the unique element $g\in W^1_{2,\mu}(a,b)$ such that $g=f$
$\mu$-a.e.
\medskip

In order to describe the operator associated with the form defined in the
following section we need some additional notions and facts concerning 
derivatives with respect to $\mu$.

If $f\in L_{1,{\rm loc}}(a,b)$, 
$g\in L_1([a,b],\mu)$ are such that $f'=g\mu$ 
(where $f'=\partial f$ denotes the distributional derivative of $f$), then
we call $g$ \emph{distributional derivative of $f$ with respect to $\mu$}, and
we
write
\[
\partial_\mu f:=g.
\]
Note that then necessarily $f'=0$ on $[a,b]\setminus\spt\mu$, i.e., $f$ is
constant on each of the components of $[a,b]\setminus\spt\mu$.
It is easy to see that this definition is equivalent to
\begin{align}\label{eq11}
f(x)=c+\int_{(a,x)}g(y)\,d\mu(y)\quad \text{a.e.},
\end{align}
with some $c\in\K$. 
Thus, the function $f$ has representatives of bounded variation and these have
one-sided limits (not depending on the representative) at all points of $[a,b]$.

\section{The form on the graph}
\label{singDir}

Let $\Gamma = (V,E,\vtx)$ be a finite directed graph. 
This means that $V$ and $E$ are finite
sets, $V\cap E=\varnothing$, $V$ is the set of \emph{vertices} (or \emph{nodes})
of
$\Gamma$, $E$ the set of \emph{edges}, and $\vtx=(\vtx_0,\vtx_1)\from E\to
V\times V$ associates with each edge $e$ a ``starting vertex'' $\vtx_0(e)$, and
an ``end vertex'' $\vtx_1(e)$.

We assume that each edge $e\in E$ corresponds to an interval $[a_e,b_e]\sse\R$
(where $a_e,b_e\in\R$, $a_e<b_e$), and we assume that $\mu_e$ is a 
finite Borel measure on $[a_e,b_e]$ satisfying either $\mu_e=0$ or else
$a_e,b_e\in\spt\mu_e$, $\mu_e(\{a_e,b_e\})=0$. We denote
\[
 E_0:=\sset{e\in E}{\mu_e=0},\qquad E_1:=E\setminus E_0.
\]

We further assume that, for each $v\in V$, we are given a weight $\mu_v\ge0$,
and we define
\[
 V_0:=\set{v\in V}{\mu_v=0},\qquad V_1:=V\setminus V_0.
\]

\begin{remark}
The sets $E_1$ and $V_1$ encode
the parts of the graph $\Gamma$, where a particle driven by the
diffusion can be localised.
In the present section we describe general glueing conditions which do not take
into account the correspondence of the edges to the vertices. In the case 
$E_1 = E$, $V_1=\varnothing$ and $\mu_e$ the Lebesgue measure on
$[a_e,b_e]$, the model will describe quantum graphs; cf.\ \cite{kps08},
\cite{kuc04}, \cite{kuc08}.
In the case $E_1 = \varnothing$ we obtain (weighted) discrete diffusion on
the vertices; cf.\cite{chlu06}.
\end{remark}

We are going to describe the self-adjoint operator driving the evolution in the
Hilbert space
\[
 \cHg:=\cHe\oplus \K^{V_1},
\]
where on
\[
 \cHe:=\bigoplus_{e\in E_1}L_2([a_e,b_e],\mu_e)
\]
we use the scalar product
\[
\scpr((f_e)_{e\in E_1},(g_e)_{e\in E_1})_\cHg:=\sum_{e\in
E_1}\scpr(f_e,g_e)_{L_2([a_e,b_e],\mu_e)},
\]
and on $\K^{V_1}$ we use the scalar product
\[
\scpr((f_v)_{v\in V_1},(g_v)_{v\in V_1})_\cHg:= \sum_{v\in V_1}
f_v\ol{g_v}\,\mu_v
\]
(for $f=\big((f_e)_{e\in E_1},(f_v)_{v\in V_1}\big),g=\big((g_e)_{e\in
E_1},(g_v)_{v\in V_1}\big)\in\cHg$).

In the following, the mapping $\iota$ defined in Section \ref{odp} will be
applied in the situation of the
edges $e\in E_1$, and will then be denoted by $\iota_e$.
We then define the operator $\iota$ from $\cHg$ to
$
 \prod_{e\in E_1}W^1_{2,\mu_e}(a_e,b_e)\times \K^{V_1},
$
by
\begin{align*}
 D(\iota)&:=\set{f\in\cHg}{f_e\in D(\iota_e)\ (e\in E_1)},\\
(\iota f)_e&:=\iota_ef_e\qquad(e\in E_1),\\
(\iota f)_v&:=f_v         \qquad(v\in V_1).
\end{align*}

We define the \emph{trace mapping} (or \emph{boundary value
mapping}) $\trace\colon\prod_{e\in E_1}C[a_e,b_e]\times\K^{V_1} \to\KEV$, where
$E_1':=E_1\times\{0,1\}$, by
\begin{align*}
 \trace f(e,j)&:=
\begin{cases}
 f_e(a_e)  &\text{if }e\in E_1,\ j=0,\\
 f_e(b_e)  &\text{if }e\in E_1,\ j=1,
\end{cases}\\
 \trace f(v)&:= f_v\qquad(v\in V_1).
\end{align*}
The space $\KEV$ will be provided with the scalar product
\[
 \scpr(\xi,\eta):=\sum_{(e,j)\in E_1'}\xi(e,j)\ol{\eta(e,j)}
+\sum_{v\in V_1}\xi(v)\ol{\eta(v)}\,\mu_v.
\]

For the definition of the form we assume that $X$ is a subspace of
$\KEV$ and that $L$ is a self-adjoint operator in $X$. Then we define the
form
$\tau$ by
\begin{align*}
 D(\tau)&:=\set{f\in D(\iota)}{\trace(\iota f)\in X},\\
\tau(f,g)&:=\sum_{e\in
E_1}\int_{a_e}^{b_e}(\iota_ef_e)'(x)\ol{(\iota_eg_e)'(x)}\,dx
           +\bigscpr(L\trace(\iota f),{\trace(\iota g)}).
\end{align*}

\begin{remark}
The subspace $X$ encodes boundary conditions for the elements of $D(\tau)$. One
would expect boundary conditions to be in the form of some equation for
$\trace(\iota f)$. Of course, 
if $P$ denotes the orthogonal projection from $\KEV$ onto $X^\bot$, 
then $D(\tau) = \set{f\in D(\iota)}{P\trace(\iota f)=0}$.

Further boundary conditions for the elements of the associated operator $H$ are
encoded in the operator $L$; we refer to the description of $H$ in
Theorem~\ref{thm-operator}.
\end{remark}

\begin{lemma}\label{form-elem}
 The form $\tau$ defined above is symmetric. $D(\tau)$ is dense if and only if 
\begin{align}\label{dense}
 \proj_{V_1}(X)=\K^{V_1},
\end{align}
where $\proj_{V_1}$ denotes the canonical projection
$\proj_{V_1}\colon\KEV\to\K^{V_1}$.
\end{lemma}

\begin{proof}
The symmetry of $\tau$ is obvious.

Assume that $D(\tau)$ is dense. The image of the dense set $D(\tau)$ under the
orthogonal projection 
\[
 \proj_2\colon \cHg\to\K^{V_1}
\]
is dense in $\K^{V_1}$, and therefore is equal to $\K^{V_1}$. From the
definition of $D(\tau)$ it follows that $\proj_2(D(\tau))$ is contained in
$\proj_{V_1}(X)$, and therefore $\proj_{V_1}(X)=\K^{V_1}$.

Now assume that \eqref{dense} holds. For $v\in V_1$ let $\xi^v\in X$ be such
that
$\xi^v(v)=1$ and $\xi^v(w)=0$ for all $w\in V_1\sms\{v\}$. Let $g^v\in D(\iota)$
be
defined by $\trace(\iota g^v)=\xi^v$, and $g^v$ affine linear on the edges. The
affine linear interpolation of the prescribed boundary values evidently
yields an element of $g^v\in D(\tau)$. 

Let $f\in\cHg$, and define
\[
 \tilde f:=f-\sum_{v\in V_1}f_vg^v.
\]
Then $\tilde f_v=0$ for all $v\in V_1$. Because $C_\rmc^1(a_e,b_e)$ is dense in
$L_2([a_e,b_e],\mu_e)$ ($e\in E_1$), the function $\tilde f$ can be
approximated by functions in
\[
 D_\rmc:=\set{f\in D(\tau)}{f_e\in C_\rmc^1(a_e,b_e)\ (e\in E_1),\ f_v=0\ (v\in
V_1)}.
\]
Therefore $f$ can be approximated by functions in
\[
D_\rmc+\sum_{v\in V_1}f_vg^v\sse D(\tau). \qedhere
\]
\end{proof}

\begin{remarks}\label{form-rems}
(a) For the special case that $X=\KEV$ and $L=0$ we denote the
corresponding form by $\tau_{\rm N}$ (the index ${\rm N}$ indicating Neumann
boundary conditions). The form $\tau_{\rm N}$ decomposes as the sum of the
Neumann forms on each of the edges and the null form on $\K^{V_1}$.
Therefore the closedness of $\tau_{\rm N}$ follows from the closedness in the
one-dimensional cases; cf.\ \cite[Section 1 and Remark 3.2]{kkvw09}.

(b) Condition \eqref{dense} did not occur in the previous treatment
\cite{kkvw09}. The reason is that it is obviously satisfied if the vertices do
not have masses, i.e.~$V_1=\varnothing$. Also, in the case of vertices with
masses, but with local boundary conditions of continuity (see
Example~\ref{lastex}), condition \eqref{dense} is automatically satisfied.
\end{remarks}

\begin{theorem}\label{form-closed}
 The form $\tau$ defined above is bounded below and closed.
\end{theorem}

\begin{proof}
For $f\in D(\tau)$ we obtain the estimate
\[
 \tau(f)=\tau_{\rm N}(f)+\bigscpr(L\trace(\iota f),{\trace(\iota f)})
\ge \tau_{\rm N}(f) - \|L\||\trace(\iota f)|^2
\]
(with $\tau_{\rm N}$ defined in Remark~\ref{form-rems}(a)).
From 
inequality \eqref{sob-ine2} we obtain that the mapping
$f\mapsto\trace(\iota f)\restrict{E_1'}$ is infinitesimally form small with
respect to $\tau_{\rm N}$. The remaining part of the trace,
$f\mapsto\trace(\iota f)\restrict{V_1}$, is bounded. These observations imply
that $\tau$ is bounded below and the that the embedding $D_{\tau_{\rm N}}\ni
f\mapsto\iota f\in(\prod_{e\in E_1}C[a_e,b_e]\times\K^{V_1},\|\cdot\|_\infty)$
is continuous. (Here, $D_{\tau_{\rm
N}}$ denotes $D(\tau_{\rm N})$, provided with the form norm.)

In order to obtain that $\tau$ is closed it is sufficient to show that
$D(\tau)$ is a closed subset of $D_{\tau_{\rm N}}$. This, however, is
immediate from the continuity of the mapping $D_{\tau_{\rm N}}\ni
f\mapsto\trace(\iota f)\in\KEV$ (and the fact that $X$ is a closed
subspace of $\KEV$).
\end{proof}

\section{The operator $H$ associated with the form $\tau$}
\label{operator}

We assume that the notation and the hypotheses are as in the previous section,
and that \eqref{dense} holds.

Besides the trace mapping defined in the previous section we also need the
\emph{signed trace} (or \emph{signed boundary values}) 
\[
 \strace\colon\prod_{e\in E_1}\bv(a_e,b_e)
\to \K^{E_1'}\sse \KEV
\]
(where $\bv(a_e,b_e)$ denotes the set of functions of bounded variation, with
equivalence of functions coinciding a.e.),
defined by
\[
 \strace f(e,j):=
\begin{cases}
 f_e(a_e\rlim) &\text{if } e\in E_1,\ j=0,\\
 -f_e(b_e\llim) &\text{if } e\in E_1,\ j=1.
\end{cases}
\]
The inclusion $\K^{E_1'}\sse \KEV$ is to be understood in the canonical sense;
we want to be able to use $\strace f$ also as an element of $\KEV$.

For the description of the self-adjoint operator $H$ associated with the form
$\tau$ we use a maximal operator $\hat H$ for the differential part of the
form. With the notation described in Section \ref{odp}, we define
\begin{align*}
&D(\hat H):=\set{f\in\smash{\prod_{e\in E_1}}D(\iota_e)}
  {(\iota_ef_e)'\in L_1(a_e,b_e),\ \partial_{\mu_e}(\iota_ef_e)'\text{
exists},\\
 &\hphantom{D(\hat H):=f\in\prod_{e\in E_1}D(\iota_e}\qquad\qquad
\partial_{\mu_e}(\iota_ef_e)'\in L_2([a_e,b_e],\mu_e)\
(e\in E_1)},\\
&\hat Hf:=(-\partial_{\mu_e}(\iota_ef_e)')_{e\in E_1}\qquad(f\in D(\hat H)).
\end{align*}
Thus, for $f\in D(\hat H)$, the signed trace $\strace ((\iota_ef_e)')_{e\in
E_1}$ exists, and it describes the ``ingoing derivatives'' from the endpoints
of the intervals. It is to be understood that for $(\iota_ef_e)'$ we choose
representatives of bounded variation (which exist by the
explanation given at the end of Section \ref{odp}), in order to be able to apply
the signed trace mapping.

Let
\[
 X_0:=\set{x\in X}{\proj_{V_1}x=0},
\]
which could also be expressed as $X_0:=X\cap \K^{E_1'}$ (with our understanding
of $\K^{E_1'}$ as a subspace of $\KEV$), and let $Q_0$ be the
orthogonal projection from $\KEV$ onto $X_0$. Also, for $v\in V_1$, let
$\xi^v\in X$ be such
that $\xi^v\restrict{V_1}=\1_{\{v\}}$ (see the proof of Lemma~\ref{form-elem}).

In the following, for $f\in D(\iota)$ we will use the shorthand notation
$(\iota f)':=\big((\iota_ef_e)'\big)_{e\in E_1}$.

\begin{theorem}\label{thm-operator}
 The operator $H$ associated with the form $\tau$ is given by
\begin{align*}
 &D(H)=\set{f\in\cHg}{(f_e)_{e\in E_1}\in D(\hat H),\ \trace (\iota f)\in X,\\
 &\hphantom{D(H)=f\in\cHg{(f_e)_{e\in E_1}\in D(\hat H)}}
Q_0\strace(\iota f)'=Q_0L\trace(\iota f)},\\
 &((Hf)_e)_{e\in E_1}=\hat H(f_e)_{e\in E_1},\\
 &(Hf)_v=\frac1{\mu_v}\bigscpr(L\trace(\iota f)-\strace(\iota f)',\xi^v)
\qquad(v\in V_1).
\end{align*}
\end{theorem}

\begin{proof}
 (i) A preliminary step: Let $f\in D(\hat H)$, $g\in D(\tau)$.
For all $e\in E_1$ one has
\begin{align*}
\int_{a_e}^{b_e}(\iota_ef_e)'(x)\ol{(\iota_eg_e)'(x)}\,dx
=&-\int_{(a_e,b_e)}\partial_{\mu_e}(\iota_ef_e)'(x)\ol{g_e(x)}\,
d\mu_e(x)\\
&+ (\iota_ef_e)'(b_e\llim)\ol{\iota_eg_e(b_e)}
- (\iota_ef_e)'(a_e\rlim)\ol{\iota_eg_e(a_e)};
\end{align*}
cf. \cite[equ.\,(1.2)]{kkvw09}. Summing this equation over all the edges in
$E_1$ we obtain
\begin{align*}
\sum_{e\in E_1}\int_{a_e}^{b_e}(\iota_ef_e)'(x)\ol{(\iota_eg_e)'(x)}\,dx
=\scpr(\hat Hf,g)_\cHe - \bigscpr(\strace((\iota_ef_e)')_{e\in
E_1},\trace{(\iota
g)})_{\K^{E_1'}}.
\end{align*}

(ii) Let $f\in D(H)$, $g\in D(\tau)$. From $D(H)\sse D(\tau)$ we conclude that
$\trace(\iota f)\in X$. As in \cite[proof of Theorem 1.9]{kkvw09} one obtains
that $(f_e)_{e\in E_1}\in D(\hat H)$, $\hat H(f_e)_{e\in
E_1}=\big((Hf)_e\big)_{e\in E_1}$. Using part (i) above we obtain
\begin{align*}
 &\scpr(Hf,g)_{\cHg}\\
 &=-\sum_{e\in E_1}\int_{(a_e,b_e)}\partial_{\mu_e}(\iota_ef_e)'(x)
          \ol{g_e(x)}\,d\mu_e(x)+ \sum_{v\in V_1}(Hf)_v\ol{g_v}\,\mu_v\\
 &=\sum_{e\in E_1}\int_{a_e}^{b_e}(\iota_ef_e)'(x)
\ol{(\iota_eg_e)'(x)}\,dx
+ \bigscpr(\strace(\iota f)',\trace{(\iota g)})_{\K^{E_1'}}
+ \sum_{v\in V_1}(Hf)_v\ol{g_v}\,\mu_v.
\end{align*}
Because of
\begin{align*}
 \scpr(Hf,g)_{\cHg}=\sum_{e\in E_1}\int_{a_e}^{b_e}(\iota_e f_e)'(x)
\ol{(\iota_e g_e)'(x)}\,dx
+\bigscpr(L\trace(\iota f),\trace(\iota g{)})
\end{align*}
we therefore obtain
\begin{align}\label{hhh}
\sum_{v\in V_1}(Hf)_v\ol{g_v}\,\mu_v=
\bigscpr(L\trace(\iota f)-\strace(\iota f)',\trace(\iota g{)}).
\end{align}

For $\xi\in X_0$ choose $g\in D(\tau)$ satisfying $\trace(\iota g)=\xi$,
and $g$ affine linear on the edges $e\in E_1$. Then equation \eqref{hhh} implies
\[
 0=\bigscpr(L\trace(\iota f)-\strace(\iota f)',\xi).
\]
This shows that $Q_0L\trace(\iota f)=Q_0\strace(\iota f)'$.

Let $v\in V_1$, and choose $g\in D(\tau)$ satisfying $\trace(\iota g)=\xi^v$,
and
$g$ affine linear on the edges $e\in E_1$. Then equation \eqref{hhh} yields
\begin{align*}
(Hf)_v\,\mu_v=\sum_{w\in V_1}(Hf)_w\ol{\xi^v(w)}\,\mu_w
             =\bigscpr(L\trace(\iota f)-\strace(\iota f)',\xi^v).
\end{align*}
This shows the second part of the formula for $Hf$.

(iii) Now let $\tilde H$ denote the operator indicated on the right hand side of
the assertion, and let $f\in D(\tilde H)$. Then $f\in D(\tau)$. Let $g\in
D(\tau)$. Then $\xi:=\trace(\iota g)-\sum_{v\in V_1}g_v\xi^v\in X_0$, and
therefore
\[
\bigscpr(L\trace(\iota f)-\strace(\iota f)',\xi)=\bigscpr(Q_0(L\trace(\iota
f)-\strace(\iota f)'),\xi)=0.
\]

Using part (i) as well as the previous equality we obtain
\begin{align*}
 &\scpr(\tilde Hf,g)_\cHg\\
 &=\bigscpr(\hat H(f_e)_{e\in E_1},{(g_e)}_{e\in E_1})_\cHe
   +\sum_{v\in V_1}\frac1{\mu_v}\scpr(L\trace(\iota f)-\strace(\iota f)',
                                       \xi^v)\ol{g_v}\,\mu_v\\
 &=\sum_{e\in E_1}\int_{a_e}^{b_e}(\iota_ef_e)'(x)\ol{(\iota_eg_e)'(x)}\,dx
+\bigscpr(\strace(\iota f)',\trace{(\iota g)})\\
&\phantom{=\bigscpr(\hat H(f_e)_{e\in E_1},{(g_e)}_{e\in
E_1})_\cHe}+\scpr(L\trace(\iota f)-\strace(\iota f)',\trace{(\iota
g)})\\
&=\tau(f,g).
\end{align*}
The definition of $H$ then yields that $f\in D(H)$ and $Hf=\tilde Hf$.
\end{proof}

\begin{remarks}
(a) For $f\in D(H)$ and $v\in V_1$, the expression given for $(Hf)_v$ given in
Theorem~\ref{thm-operator} does not depend on the choice of $\xi^v$.

(b) The case of a weight $\mu_v>0$ at a vertex leads to a case of Wentzell
boundary condition at $v$. The expression of $(Hf)_v$ in
Theorem~\ref{thm-operator} generalises the expression obtained at a boundary
point in the case of a single interval; cf.\ \cite[Poposition~4.3]{vovo03}.
\end{remarks}

\section{Positivity and contractivity}
\label{pos-con}

In this section we indicate conditions for the $C_0$-semigroup
$(e^{-tH})_{t\ge0}$ to be positive or submarkovian. We assume that the
hypotheses are as in Section \ref{singDir} and that \eqref{dense} holds.

In the following we need the notion of a (Stonean) sublattice of $\K^n$. We
consider $\K^n$ as the function space $C(\{1,\dots,n\})$, and accordingly use
the notation
$|x|=(|x_1|,\dots,|x_n|)$, for $x\in\K^n$, and $x\wedge y=(x_1\wedge y_1,\dots,
x_n\wedge y_n)$, for $x,y\in\R^n$. A \emph{sublattice} $X$ of $\K^n$ is a
subspace for which $x\in X$ implies that $|x|\in X$. A sublattice $X$ is called 
\emph{Stonean} if additionally 
$x\wedge\indic\in X$ for all real $x\in X$.

We refer to \cite[Appendix]{kkvw09} for the description
of (Stonean) sublattices of $\K^n$ and 
of generators for positive (submarkovian) $C_0$-semigroups
on these sublattices.

\begin{theorem}\label{pos-subm}
{\rm (a)} Assume that $X$ is a sublattice of $\KEV$ and that the semigroup
$(e^{-tL})_{t\ge0}$ is positivity preserving. Then $(e^{-tH})_{t\ge0}$ is
positivity preserving.

{\rm (b)} Assume that $X$ is a Stonean sublattice of $\KEV$ and that the
semigroup $(e^{-tL})_{t\ge0}$ is a submarkovian semigroup. Then
$(e^{-tH})_{t\ge0}$ is submarkovian.
\end{theorem}

This result was proved for the case of local boundary conditions
(cf.~Section~\ref{lbc}) and no vertex masses in \cite[Theorem
3.5]{kkvw09}, and for the case of vertices with masses and local boundary
conditions of continuity (cf.~Example~\ref{lastex}) in 
\cite[Theorem 4.2]{kkvw09}. Its proof is completely analogous to 
\cite[proof of Theorem 3.5]{kkvw09}; so we refrain from giving a complete proof
but rather only  mention the main ingredients. 
The proof 
consists in an application of the
Beurling-Deny criteria
(cf.~\cite[Corollary 2.18]{ouh05}; see also \cite[Remarks 1.6]{kkvw09}).
So, in order to prove part (a), it is equivalent to
prove that the normal contraction $f\mapsto |f|$ acts on $D(\tau)$, and that
$\tau(|f|)\le\tau(f)$ for all $f\in D(\tau)$. That the inequality works on the
differential part is a one-dimensional issue which is taken care of in
\cite[Theorem 1.7]{kkvw09}. For the trace part, the main observation is the
equation $\trace \iota |f|=|\trace\iota f|$.
This is less obvious than it might
appear at the first glance since, in general, one does not have
$\iota|f|=|\iota f|$. However, this equality holds on $\spt\mu_e$, and
therefore 
at 
the end points of the intervals $[a_e,b_e]$, for all
$e\in E_1$. The reasoning for part (b) is analogous.

\section{Local boundary conditions}
\label{lbc}

So far, the structure of the graph did not enter the considerations; in fact
the function $\vtx$ linking the edges to the vertices was not used at all. In
order to explain what we understand by local boundary conditions, we need the
following definitions.

For $v\in V$, the sets
\[
 E_{1,v,j}:=\set{e\in E_1}{\vtx_j(e)=v}\qquad (j=0,1)
\]
describe the sets of all edges having mass and starting or ending at $v$,
respectively, and the set
\[
 E_{1,v}:=\big(E_{1,v,0}\times\{0\}\big)\cup\big(E_{1,v,1}\times\{1\}\big)
\]
is the set of all edges having mass connected with $v$ (and where loops starting
and ending
at $v$ yield two contributions). Note that then $E_1'=\bigcup_{v\in V}E_{1,v}$.

Recall that the boundary conditions are specified by the choice of a subspace
$X\sse\KEV$ and a self-adjoint operator $L$ in $X$.
The boundary conditions will be called \emph{local} if for each $v\in V$ there
exists a subspace
\begin{align*}
 X_v\sse\K^{E_{1,v}}\quad\text{if }v\in V_0,\qquad
 X_v\sse\K^{E_{1,v}\cup\{v\}}\quad\text{if }v\in V_1,
\end{align*}
and a selfadjoint operator $L_v$ in $X_v$, such that
\[
 X=\bigoplus_{v\in V}X_v,\qquad L=\bigoplus_{v\in V}L_v.
\]

For $v\in V$, we define the ``local trace mapping''
\begin{align*}
\trace_v\colon\prod_{e\in E_1}C[a_e,b_e]\times\K^{V_1} \to
\begin{cases}
 \K^{E_{1,v}}    &\text{if } v\in V_0,\\
 \K^{E_{1,v}\cup\{v\}}  &\text{if } v\in V_1
\end{cases}
\end{align*} 
by 
\[
\trace_vf:=
\begin{cases}
\trace f\restrict{E_{1,v}}&\text{if }v\in V_0,\\
\trace f\restrict{E_{1,v}\cup\{v\}}&\text{if }v\in V_1. 
\end{cases}
\]
Then for the form $\tau$ we obtain
\begin{align*}
D(\tau)&=\set{f\in D(\iota)}{\trace_v(\iota f)\in X_v\ (v\in V)},\\
\tau(f,g)&=\sum_{e\in
E_1}\int_{a_e}^{b_e}(\iota_ef_e)'(x)\ol{(\iota_eg_e)'(x)}\,dx
           +\sum_{v\in V}\bigscpr(L_v\trace_v(\iota f),{\trace_v(\iota g)}). 
\end{align*}
With
\[
 X_{v,0}:=
\begin{cases}
 X_v &\text{if }v\in V_0,\\
 \sset{\xi\in X_v}{\xi(v)=0} &\text{if }v\in V_1, 
\end{cases}
\]
the condition \eqref{dense} for $D(\tau)$ to be dense then decomposes into
\begin{align*}
 X_{v,0}\ne X_v\qquad(v\in V_1),
\end{align*}
or expressed differently,
for all $v\in V_1$ there exists $\xi^v\in X_v$ such that $\xi^v(v)=1$.

It is an easy task to translate the description of the associated operator $H$,
given in Theorem \ref{thm-operator}, to the present case of local boundary
conditions, as follows.

\begin{corollary}\label{cor-local1}
The operator $H$ associated with $\tau$ is given by
\begin{align*}
 &D(H)=\set{f\in\cHg}{(f_e)_{e\in E_1}\in D(\hat H),\ \trace_v(\iota f)\in
X_v,\\
 &\hphantom{D(H)=f\in\cHg{(f_e)_{e\in E_1}}}
Q_{v,0}\strace_v(\iota f)'=Q_{v,0}L_v
\trace_v(\iota f)\ (v\in V)},\\
 &((Hf)_e)_{e\in E_1}=\hat H(f_e)_{e\in E_1},\\
 &(Hf)_v=\frac1{\mu_v}\bigscpr(L_v\trace_v(\iota f)-\strace_v(\iota f)',\xi^v)
\qquad(v\in V_1).
\end{align*}
\end{corollary}

Here, for $v\in V$ the mapping $\strace_v\colon\prod_{e\in E_1}\bv(a_e,b_e)
\to \K^{E_{1,v}}$ is defined by $\strace_vf:=(\strace f)\restrict{E_{1,v}}$, 
and $Q_{v,0}$ is the orthogonal
projection onto $X_{v,0}$ in $\K^{E_{1,v}}$, for $v\in V_0$, or
in $\K^{E_{1,v}\cup\{v\}}$, for $v\in V_1$.
We will not put down further details here. Similarly, the conditions
for $(e^{-tH})_{t\ge0}$ to be positive and submarkovian, Theorem \ref{pos-subm},
can be spelled out in terms of the spaces $X_v$ and the operators $L_v$. 
The
statements are then analogous to \cite[Theorem 3.5]{kkvw09}, where the
case that $E=E_1$ and $V=V_0$ is treated.

\begin{corollary}\label{cor-local2}
{\rm (a)} Assume that $X_v$ is a sublattice of $\K^{E_{1,v}}$ ($v\in V_0$)
or $\K^{E_{1,v}\cup\{v\}}$ ($v\in V_1$)
and that $(e^{-tL_v})_{t\geq0}$ positivity preserving, for all $v\in V$. 
Then $(e^{-tH})_{t\geq0}$ is a positivitiy preserving 
$C_0$-semigroup on $\cHg$.

{\rm (b)} Assume that $X_v$ is a Stonean sublattice of $\K^{E_{1,v}}$ ($v\in
V_0$) 
or $\K^{E_{1,v}\cup\{v\}}$ ($v\in V_1$) and that $(e^{-tL_v})_{t\geq0}$ is a submarkovian 
$C_0$-semigroup on $X_v$, for all $v\in V$. Then $(e^{-tH})_{t\geq0}$ is a 
submarkovian $C_0$-semigroup on $\cHg$.
\end{corollary}

\begin{example}[local boundary conditions of continuity]\label{lastex}
This special case of local boundary conditions was studied in 
\cite[Section 4]{kkvw09}. In our framework, this example reads as follows.
Let $X_v = \lin 
\{\1\}$, $L_v\in\R$, $l_v:= \scpr(L_v \1,\1)$ for $v\in V$. 
Then $X_{v,0} = \{0\}$ for $v\in V_1$ 
(which makes it clear that condition \eqref{dense} is satisfied) and hence
\[
Q_{v,0} = 
\begin{cases}
 \scpr(\cdot,\1)\1 
 &\text{if }v\in V_0,\\
 0 &\text{if }v\in V_1.
\end{cases}
\]
Functions $f\in D(\tau)$ are continuous
on $\Gamma$,
i.e., for $v\in V_1$ we 
have $f(v) = (\trace_v f)(e,j)$ for all $(e,j)\in E_{1,v}$, 
and for $v\in V_0$ 
there exists 
$a_v(f)\in \K$ such that $a_v(f) = (\trace_v f)(e,j)$ for all 
$(e,j)\in E_{1,v}$ (note that we cannot write $f(v)$ in this case since $f$ 
is not defined on $V_0$).
The 
second part of the 
boundary conditions for $f\in D(H)$ translates to
\[
\sum_{e\in E_{1,v,0}} f_e'(a_e\rlim) - \sum_{e\in E_{1,v,1}} f_e'(b_e\llim) = 
l_v a_v(f) \quad(v\in V_0);
\]
see also \cite[Theorem 4.3]{kkvw09}. 
In the setup considered in \cite{kuc04}, these boundary conditions are called 
$\delta$-type conditions; cf.~\cite[Section 3.2.1]{kuc04}.
\end{example}

{
}

\bigskip
\noindent
Christian Seifert\\
Fakult\"at Mathematik\\
Technische Universit\"at Chemnitz\\
09107 Chemnitz, Germany \\
{\tt christian.seifert@mathematik.tu-chemnitz.de}
\\[3ex]
J\"urgen Voigt\\
Fachrichtung Mathematik\\
Technische Universit\"at Dresden\\
01062 Dresden, Germany\\
{\tt juergen.voigt@tu-dresden.de}
\end{document}